\documentclass{amsart}

\usepackage{amsmath}
\usepackage{hyperref}
\usepackage[english]{babel}
\usepackage[utf8]{inputenc}
\usepackage{amsmath}
\usepackage{amsfonts}
\usepackage{amsthm}
\usepackage{todonotes}

\newcommand{\N}{\mathbb{N}}
\newcommand{\R}{\mathbb{R}}
\newcommand{\Z}{\mathbb{Z}}
\newcommand{\T}{\mathbb{T}}
\newcommand{\eps}{\varepsilon}
\newcommand{\dd}{\, \mbox{d}}

\newtheorem{lm}{Lemma}
\newtheorem{thm}{Theorem}
\theoremstyle{definition}
\newtheorem{definition}{Definition}
\newtheorem{remark}{Remark}

\title{Generalized Bounded Distortion Property}

\author{Gregory Borissov}

\address{Gregory Borissov, Department of Mathematics, University of California, Irvine, CA~92697, USA}

\email{gborisso@uci.edu}

\author{Grigorii Monakov}

\address{Grigorii Monakov, Department of Mathematics, University of California, Irvine, CA~92697, USA}

\email{gmonakov@uci.edu}

\begin{document}
\maketitle

\begin{abstract}
    We prove the Nonstationary Bounded Distortion Property for $C^{1 + \varepsilon}$ smooth dynamical systems on multidimensional spaces. The results we obtain are motivated by potential application to study of spectral properties of discrete Schr\"odinger operators with potentials generated by Sturmian sequences.
\end{abstract}

\section{Introduction}

The Bounded Distortion Lemma is a well-known and widely used tool in the theory of smooth dynamical systems. One of its classical forms is the following:

\begin{thm}
    Consider a twice continuously differentiable function $f: \R \to \R$ with a non-vanishing derivative, such that
    \begin{equation*}
        \frac{|f''(x)|}{|f'(x)|} < C
    \end{equation*}
    for some constant $C$ and any $x \in \R$. Assume that there exists an interval $I \subset \R$, such that for some $n \in \mathbb{N}$ we have
    \begin{equation*}
        \sum_{j = 0}^{n - 1} |f^j(I)| < L
    \end{equation*}
    for some constant $L$. Then there exists a constant $K$ that depends only on $L$ and $C$ (and not $n$), such that for any two subintervals $I_1, I_2 \subset I$,
    \begin{equation*}
        K^{-1} \frac{|I_1|}{|I_2|} < \frac{|f^n(I_1)|}{|f^n(I_2)|} < K \frac{|I_1|}{|I_2|}.
    \end{equation*}
\end{thm}

Different forms of this statement are used in \cite{De, L, PT, Y, F} to establish various fundamental results.

One of its most famous applications is in the proof of Denjoy's Theorem, which states that any diffeomorphism of the circle with irrational rotation number and derivative of bounded variation is topologically conjugate to an irrational rotation (for details see \cite{De} or \cite[Theorem 12.1.1]{KH}). 

Another application of bounded distortion can be seen in the proof of a classical folklore result that asserts existence of a continuous invariant ergodic measure for $C^2$ expanding maps of the circle. Moreover, if a piecewise differentiable map admits an induced Markov map and satisfies the Volume Bounded Distortion Property (also known as geometrical self similarity) and some additional assumptions, then it admits an ergodic invariant absolutely continuous probability measure. For details and other results of similar nature see \cite{L} and references therein.

Bounded distortion plays an important role in the study of dynamically defined Cantor sets. For example, it is used to prove that the Hausdorff dimension coincides with the box counting dimension (also known as limit capacity) if the generators are in $C^{1 + \eps}$ (see \cite{T}, or \cite[Chapter 4, Theorem 3]{PT} for a modern exposition). Bounded distortion also allows to estimate thickness of a dynamically defined Cantor set, and that estimate is used to prove the celebrated Newhouse Theorem (see \cite{N1} and \cite{PT} for details).

A special case of bounded distortion for nonstationary sequences of uniformly hyperbolic maps of a plane with the uniform cone condition was studied by J. Palis and J-C. Yoccoz in \cite[Corollary 3.4]{PY}. A generalization of said result that allows unbounded derivatives can be found in \cite{N2}. These results have important applications in the theory of SRB measures for surface diffeomorphisms.
%
%
For more results utilizing the bounded distortion technique see \cite{Y, F} and references therein. 

We would like to point out that bounded distortion is a phenomenon that is usually observed in systems with $C^{1 + \eps}$ regularity (see Theorem \ref{thm:main}). For example, one could see \cite{O} by H. Ounesli, which proves that within the space of ergodic Lebesgue-preserving $C^1$ expanding maps of the circle, unbounded distortion is $C^1$-generic.


The main motivation for our result is a potential application to the study of spectral properties of discrete Schr\"odinger operators with Sturmian potentials. For details see Appendix \ref{app:spectral}.

\section{Main result} \label{sec:multDim}

In this section we will state and prove the main result of the paper. We start with some definitions.

\begin{definition}
    If $\gamma$ is a regular $C^1$ curve, meaning $\gamma:[a,b] \subset \R \rightarrow \R^d$ and $\gamma'(t)$ is continuous and does not vanish, then define the \textbf{length} of $\gamma$ as
    \begin{equation*}
        L(\gamma) = \int_a ^b \|\gamma'(t)\| \dd t,
    \end{equation*}
    and define the \textbf{maximal angle} of $\gamma$ as
    \begin{equation*}
        \alpha(\gamma) = \sup_{x,y \in [a,b]} \angle(\gamma'(x),\gamma'(y)) .
    \end{equation*}
\end{definition}

\begin{definition}
    For a continuously differentiable map $f: \R^n \rightarrow \R^m$
    \begin{equation*}
    f(x) = \begin{pmatrix}
        f^{(1)}(x) \\
        \vdots \\
        f^{(m)}(x)
    \end{pmatrix},
    \end{equation*}
    and a point $x \in \R^n$ define the \textbf{total derivative} of $f$ at $x$ as
    \begin{equation*}
        D_x f = \begin{pmatrix}
            \frac{\partial f^{(1)}}{\partial x_1} & \dots & \frac{\partial f^{(1)}}{\partial x_n} \\
            \vdots & \ddots & \vdots \\
            \frac{\partial f^{(m)}}{\partial x_1} & \dots & \frac{\partial f^{(m)}}{\partial x_n}
        \end{pmatrix}.
    \end{equation*}
\end{definition}

\begin{definition}
    For a continuously differentiable map $f: \R^n \rightarrow \R^m$, define its \textbf{$C^1$} \textbf{seminorm} as
    \begin{equation*}
        \|f\|_{1} = \sup_{x,v \in \R^n, \|v\|= 1} \|D_x f \cdot v\|,
    \end{equation*}
    and for $0 < \eps \le 1$, define its $C^{1+\eps}$ \textbf{seminorm} as
    \begin{equation*}
        \|D f\|_{\eps} = \sup_{x, y \in \R^n} \frac{\|D_x f  - D_y f\|}{\|x - y\|^{\eps}}.
    \end{equation*}
    We say that $f \in C^{1 + \eps}$ if $\|D f\|_{\eps} < \infty$.
\end{definition}

\begin{remark}
    Note that the case of $\eps = 1$ corresponds to $f$ with a Lipschitz first derivative.
\end{remark}

The following Theorem is the main result of the paper:

\begin{thm} \label{thm:main}
    Let $\gamma: [a_0, b_0] \subset \R \rightarrow \R^d$ be the natural parameterization of a regular $C^1$ curve. Fix $0 < \eps \le 1$ and let $f_i:\R^d \rightarrow \R^d$ be $C^{1 + \eps}$ maps. Denote $F_i = f_i \circ f_{i-1} \circ ... \circ f_1$, where $1 \leq i \leq n$ and let $F_0$ be the identity map. Assume there exists $C$ such that 
    \begin{equation} \label{reg_multDim}
        \|f_i\|_1, \|f_i^{-1}\|_1, \|D f_i\|_{\eps} \leq C
    \end{equation}
    for all $1 \leq i \leq n$ and assume there exists $L$ and $\alpha$ such that $\sum_{i=0}^{n-1} L(F_{i} \circ \gamma)^\eps \le L$ and $\sum_{i=0}^{n-1} \alpha(F_{i} \circ \gamma) \le \alpha$. Then there exists a constant $K$ such that for any $x, y \in [a_0, b_0]$,
    \begin{equation} \label{derDist}
        K^{-1} \leq \frac{\|(F_n \circ \gamma)'(x)\|}{\|(F_n \circ \gamma)'(y)\|} \leq K,
    \end{equation}
    where $K$ depends only on $C, L,$ and $\alpha$ (and not $n$). 
    
    Moreover, consider two subintervals $[a_1, b_1],[a_2, b_2] \subset [a_0, b_0]$ and related arcs given by $\gamma([a_1, b_1])$ and $ \gamma([a_2, b_2])$. Then
    \begin{equation} \label{lengthDist}
        \frac{|b_1 - a_1|}{|b_2 - a_2|} K^{-2} \leq \frac{L((F_n \circ \gamma)([a_1, b_1]))}{L((F_n \circ \gamma)([a_2, b_2]))} \leq \frac{|b_1 - a_1|}{|b_2 - a_2|} K^2.
    \end{equation}
\end{thm}

\begin{proof}

    We start by proving inequality (\ref{derDist}). Denote points $x_i = (F_i \circ \gamma)(x),\ y_i = (F_i \circ \gamma)(y)$, and vectors $u_i = (F_i \circ \gamma)'(x),\ v_i = (F_i \circ \gamma)'(y)$ for $0 \leq i \leq n$. Notice that all $\| f_i^{-1}\|_1$ are finite, so we can say that $\|u_i\|, \|v_i\| \neq 0$ for all $1 \leq i \leq n$. Now we can manipulate the following equation:
    \begin{equation} \label{biglog}
    \begin{split}
        \left| \log \frac{\|u_n\|}{\|v_n\|} \right| &= \left| \log \frac{\|D_{x_{n-1}}f_n \cdot u_{n-1}\|}{\|D_{y_{n-1}}f_n \cdot v_{n-1}\|} \right| \leq \\
            &\leq \left| \log \frac{\|D_{x_{n-1}}f_n \cdot u_{n-1}\|}{\|D_{x_{n-1}}f_n \cdot v_{n-1}\|} \right| + \left| \log \frac{\|D_{x_{n-1}}f_n \cdot v_{n-1}\|}{\|D_{y_{n-1}}f_n \cdot v_{n-1}\|} \right|. \\
    \end{split}
    \end{equation}
    We continue by bounding both logarithms, and we will do so with the following lemmas:

    \begin{lm} \label{Lemma 1}
        For every $1 \leq i \leq n$,
        \begin{equation*}
            \left| \log \frac{\|D_{x_{i-1}}f_i \cdot u_{i-1}\|}{\|D_{x_{i-1}}f_i \cdot v_{i-1}\|} \right| \leq \left| \log \frac{\|u_{i-1}\|}{\|v_{i-1}\|} \right| + C^2 \cdot \alpha(F_{i-1} \circ \gamma)
        \end{equation*}
    \end{lm}
    \begin{lm} \label{Lemma 2}
        For every $1 \leq i \leq n$,
        \begin{equation*}
            \left| \log \frac{\|D_{x_{i-1}}f_i \cdot v_{i-1}\|}{\|D_{y_{i-1}}f_i \cdot v_{i-1}\|} \right| \leq C^2 \cdot L(F_{i-1} \circ \gamma)^\eps
        \end{equation*}
    \end{lm}
    \begin{proof}[Proof of Lemma \ref{Lemma 1}]
        Let $\tilde{u}_{i-1}$ and $\tilde{v}_{i-1}$ notate the normalized vectors of $u_{i-1}$ and $v_{i-1}$ respectively. Then
        \begin{equation*}
            \left| \log \frac{\|D_{x_{i-1}}f_i \cdot u_{i-1}\|}{\|D_{x_{i-1}}f_i \cdot v_{i-1}\|} \right|
            \leq \left| \log \frac{\|D_{x_{i-1}}f_i \cdot \tilde{u}_{i-1}\|}{\|D_{x_{i-1}}f_i \cdot \tilde{v}_{i-1}\|} \right| + \left| \log \frac{\|u_{i-1}\|}{\|v_{i-1}\|} \right|.
        \end{equation*}
        Without loss of generality, assume $\|D_{x_{i-1}}f_i \cdot \tilde{u}_{i-1}\| \geq \|D_{x_{i-1}}f_i \cdot \tilde{v}_{i-1}\|$.  If we only focus on the part of the left term that is inside the logarithm, we can rearrange it as
        \begin{equation} \label{angEst1}
        \begin{split}
            \frac{\|D_{x_{i-1}}f_i \cdot \tilde{u}_{i-1}\|}{\|D_{x_{i-1}}f_i \cdot \tilde{v}_{i-1}\|} &= \frac{\|D_{x_{i-1}}f_i \cdot \tilde{u}_{i-1} - D_{x_{i-1}}f_i \cdot \tilde{v}_{i-1} + D_{x_{i-1}}f_i \cdot \tilde{v}_{i-1}\|}{\|D_{x_{i-1}}f_i \cdot \tilde{v}_{i-1}\|} \le \\
            &\leq \frac{\|D_{x_{i-1}}f_i\| \cdot \|(\tilde{u}_{i-1} - \tilde{v}_{i-1})\|}{\|D_{x_{i-1}}f_i \cdot \tilde{v}_{i-1}\|} + 1. \\
        \end{split}
        \end{equation}
        Denote $\alpha_{i-1} = \angle(\tilde{u}_{i-1}, \tilde{v}_{i-1})$. We can use the definition of $\alpha_{i-1}$ and the fact that $\sin \left( \frac{\alpha_{i-1}}{2} \right) \leq \frac{\alpha_{i-1}}{2}$ to get
        \begin{equation} \label{angEst2}
            \|(\tilde{u}_{i-1} - \tilde{v}_{i-1})\| = 2\sin \left( \frac{\alpha_{i-1}}{2} \right) \leq \alpha_{i-1} \leq \alpha(F_{i-1} \circ \gamma).
        \end{equation}
        Note that $\det (D_{x_{i-1}}f_i) \neq 0$ since every $\| f_i^{-1} \|$ is finite, hence its inverse exists. Combining inequalities (\ref{angEst1}) and (\ref{angEst2}) gives us
        \begin{equation*}
            \frac{\|D_{x_{i-1}}f_i \cdot \tilde{u}_{i-1}\|}{\|D_{x_{i-1}}f_i \cdot \tilde{v}_{i-1}\|} \leq \|D_{x_{i-1}}f_i\| \cdot \|(D_{x_{i-1}}f_i) ^{-1}\| \cdot \alpha(F_{i-1} \circ \gamma) + 1 \leq C^2 \cdot \alpha(F_{i-1} \circ \gamma) + 1.
        \end{equation*}
        Since $\left| \log (1+x) \right| \leq x$ for all $x \geq 0$, and that $|\log x|$ is monotone increasing for $x \geq 1$, we can get our final bound:
        \begin{equation} \label{logBound}
            \left| \log \frac{\|D_{x_{i-1}}f_i \cdot \tilde{u}_{i-1}\|}{\|D_{x_{i-1}}f_i \cdot \tilde{v}_{i-1}\|}  \right| \leq \left| \log \left( C^2 \cdot \alpha(F_{i-1} \circ \gamma) + 1 \right) \right| \leq C^2 \cdot \alpha(F_{i-1} \circ \gamma).
        \end{equation}
    \end{proof}
    \begin{proof}[Proof of Lemma ~\ref{Lemma 2}]
    Once again, let $\tilde{v}_{i-1}$ notate the normalized vector of $v_{i-1}$. Without loss of generality, assume $\|D_{x_{i-1}}f_i \cdot \tilde{v}_{i-1}\| \geq \|D_{y_{i-1}}f_i \cdot \tilde{v}_{i-1}\|$. Then
    \begin{equation*}
    \begin{split}
         \frac{\| D_{x_{i-1}}f_i \cdot v_{i-1} \|}{\| D_{y_{i-1}}f_i \cdot v_{i-1} \|} &= \frac{\|D_{x_{i-1}}f_i \cdot \tilde{v}_{i-1} - D_{y_{i-1}}f_i \cdot \tilde{v}_{i-1} + D_{y_{i-1}}f_i \cdot \tilde{v}_{i-1}\|}{\|D_{y_{i-1}}f_i \cdot \tilde{v}_{i-1}\|} \le \\
         &\leq C \cdot \|D_{x_{i-1}}f_i - D_{y_{i-1}}f_i\| + 1,
    \end{split}
    \end{equation*}
    and we can bound the following similar to (\ref{logBound}):
    \begin{equation*}
        \left| \log \frac{\| D_{x_{i-1}}f_i \cdot v_{i-1} \|}{\| D_{y_{i-1}}f_i \cdot v_{i-1} \|} \right| \leq C \cdot \|D_{x_{i-1}}f_i - D_{y_{i-1}}f_i\|.
    \end{equation*}
    Note that
    \begin{equation*}
        \|x_{i-1} - y_{j-1}\| \leq L((F_{i-1} \circ \gamma)([x,y])) \leq L(F_{i-1} \circ \gamma),
    \end{equation*}
    then
    \begin{equation*}
        1 \leq \frac{L(F_{i-1} \circ \gamma)^\eps}{\|x_{i-1} - y_{j-1}\|^\eps}.
    \end{equation*}
    It follows that
    \begin{equation*}
        C \cdot \|D_{x_{i-1}}f_i - D_{y_{i-1}}f_i\| \leq C \frac{\|D_{x_{i-1}}f_i - D_{y_{i-1}}f_i\|}{\|x_{i-1} - y_{j-1}\|^\eps} \cdot L(F_{i-1} \circ \gamma)^\eps \leq C^2 \cdot L(F_{i-1} \circ \gamma)^\eps.
    \end{equation*}
    \end{proof}
    Combining Lemma \ref{Lemma 1} and Lemma \ref{Lemma 2} with equation (\ref{biglog}) shows
    \begin{equation*}
        \left| \log \frac{\|u_n\|}{\|v_n\|} \right| \leq \left| \log \frac{\|u_{n-1}\|}{\|v_{n-1}\|} \right| + C^2 \cdot \alpha(F_{n-1} \circ \gamma) + C^2 \cdot L(F_{n-1} \circ \gamma)^\eps.
    \end{equation*}
    We can apply this inequality recursively to get
    \begin{equation*}
        \left| \log \frac{\|u_n\|}{\|v_n\|} \right| \leq C^2 \sum_{i=0}^{n-1} \big[ \alpha(F_{i} \circ \gamma) + L(F_{i} \circ \gamma)^\eps \big] \leq C^2 (\alpha + L).
    \end{equation*}
    We finish the proof of (\ref{derDist}) by taking $K = e^{C^2 (\alpha + L)}$.

    Now let us move on to inequality (\ref{lengthDist}). Let $t \in [a_0, b_0] $ be arbitrary. Note that $\|(F_n \circ \gamma)'(t)\| \neq 0$ since $\det (D_{\gamma(t)} f_i) \neq 0 $ and $\gamma' (t) \neq 0$. Then by definition of length we have
    \begin{equation*}
        \frac{L((F_n \circ \gamma)([a_1, b_1]))}{\|(F_n \circ \gamma)'(t)\|} = \int_{a_1}^{b_1} \frac{\| (F_n \circ \gamma)' (x) \|}{\|(F_n \circ \gamma)'(t)\|} \dd x 
    \end{equation*}
    and inequality \eqref{derDist} gives us
    \begin{equation*}
        |b_1 - a_1|K^{-1} \leq \frac{L((F_n \circ \gamma)([a_1, b_1]))}{\|(F_n \circ \gamma)'(t)\|} \le |b_1 - a_1| K.
    \end{equation*}
    We can similarly obtain
    \begin{equation*}
        \frac{K^{-1}}{|b_2 - a_2|}
        \le \frac{\|(F_n \circ \gamma)'(t)\|}{L((F_n \circ \gamma)([a_2, b_2]))}
        \le \frac{K}{|b_2 - a_2|},
    \end{equation*}
    and multiplying these two inequalities finishes the proof.
\end{proof}

\appendix
\section{Application in spectral theory} \label{app:spectral}

In this subsection we describe the main motivation behind our paper. 

Spectral properties of discrete Schr\"odinger operators with dynamically defined potentials has been an object of intensive studies during the last several decades. Let us briefly describe a model we hope to apply our results to. For a positive number $\lambda > 0$ (usually referred to as coupling constant) and irrational angle $\alpha \in \T^1 = \R/\Z$ let us define the following bounded, self-adjoint operator on $l^2(\Z)$:

\[
    (H_{\alpha, \lambda, \omega} \psi) (n) = \psi(n - 1) + \psi(n + 1) + \lambda R_{\alpha, \omega} (n) \psi(n), \quad \psi \in l^2(\Z),
\]
where 

\[
    R_{\alpha, \omega} (n) = \chi_{[1 - \alpha, 1)} (n \alpha + \omega \mod 1).
\]
$H_{\alpha, \lambda, \omega,}$ is a discrete Schr\"odinger operator with the potential $R_{\alpha, \omega} (n)$. The sequence $\{R_{\alpha, \omega} (n)\}_{n \in \Z}$ is usually referred to as rotation sequence and is known to be Sturmian, i.e. a non-periodic sequence of the lowest possible complexity (for a formal definition see \cite{Fogg}). One of canonical examples of a rotation sequence is given by $R_{\left( \frac{\sqrt{5} - 1}{2}, 0 \right)}$. This sequence can also be obtained as a limit sequence for the Fibonacci substitution: $0 \to 01$, $1 \to 0$. We will denote the corresponding Schr\"odinger operator by $H_{\lambda}$ for brevity. Spectral properties of operator $H_{\lambda}$, which is usually referred to as Fibonacci Hamiltonian, were extensively studied by many authors (see \cite{DGY, S2} and references therein).

One of the tools that proved to be extremely useful in the study of the spectrum of Fibonacci Hamiltonian is Trace Map Formalism. Let us briefly describe the mechanism behind it. Consider a map
\begin{equation} \label{T_1}
    T_1: \R^3 \to \R^3, \quad T_1(x, y, z) = (2xy - z, x, y).
\end{equation}
Map $T_1$ preserves the following family of surfaces in $\R^3$:
\[
    S_{\lambda} = \bigg\{ (x, y, z) \; \bigg| \; x^2 + y^2 + z^2 - 2xyz = 1 + \frac{\lambda^2}{4} \bigg\}.
\]
Denote by $l_{\lambda}$ the line 
\[
    l_{\lambda} = \bigg \{ \left( \frac{E - \lambda}{2}, \frac{E}{2}, 1 \right) \bigg| E \in \R \bigg\}.
\]
It's easy to check that $l_{\lambda} \subset S_{\lambda}$. The following celebrated result by A. S\"{u}t\H{o} establishes a connection between spectral properties of $H_{\lambda}$ and dynamical properties of the map $T_1$.
\begin{thm}[A. S\"ut\H{o}, 1987, \cite{S1}] \label{thm:FibHam}
    $E \in \R$ lies in the spectrum of $H_{\lambda}$ if and only if the positive semitrajectory 
    \[
        \bigg \{ T_1^{n} \left( \frac{E - \lambda}{2}, \frac{E}{2}, 1 \right) \bigg| n \in \N \bigg \} \subset S_{\lambda}
    \]
    is bounded.
\end{thm}

This concise description of the spectrum is possible because of a nice property of the golden mean $\frac{\sqrt{5} - 1}{2}$. Namely, the continuous fraction expansion of that number consists of all ones:
\[
    \frac{\sqrt{5} - 1}{2} = \cfrac{1}{1 + \cfrac{1}{1 + \cfrac{1}{1 + \ldots \vphantom{\cfrac{1}{1}}}}}.
\]
A similar description is available for an arbitrary irrational 
\begin{equation} \label{contFrac}
    \alpha = \cfrac{1}{a_1 + \cfrac{1}{a_2 + \cfrac{1}{a_3 + \ldots \vphantom{\cfrac{1}{1}}}}}.
\end{equation}
In order to give it we have to define a sequence of maps $\{T_a\}_{a \in \N}$. We start with two auxiliary maps:
\[
    U:\R^3 \to \R^3, \quad U(x, y, z) = (2xz - y, x, z)
\]
and 
\[
    P:\R^3 \to \R^3, \quad U(x, y, z) = (x, z, y).
\]
Now we define $T_1 = UP$ (which also agrees with formula (\ref{T_1})) and 
\[
    T_{a + 1} = U T_{a} = U^{a} T_1 = U^{a + 1} P.
\]
The family of surfaces $S_{\lambda}$ is preserved by $U$ and $P$, hence also preserved by every $T_a$. For a fixed irrational $\alpha$ with a continued fraction expansion given by (\ref{contFrac}) and every $k \in \N$ define
\[
    \left( x_k(E), y_k(E), z_k(E) \right) = T_{a_k} \circ T_{a_{k - 1}} \circ \ldots \circ T_{a_1} \left( \frac{E - \lambda}{2}, \frac{E}{2}, 1 \right).
\]
The description of the spectrum of $H_{\alpha, \lambda, \omega}$ is given by the following
\begin{thm}[J. Bellissard, B. Iochum, E. Scoppola, D. Testard, 1989, \cite{BIST}] \label{thm:Sturm}
    $E \in \R$ lies in the spectrum of $H_{\alpha, \lambda, \omega}$ if and only if the corresponding trajectory $\{(x_k(E), y_k(E), z_k(E))\}_{k \in \N}$ is bounded.
\end{thm}

Using the description from Theorem \ref{thm:FibHam} D. Damanik and A. Gorodetski were able to estimate the thickness and Hausdorff dimension of the spectrum of Fibonacci Hamiltonian $H_{\lambda}$ for the coupling constant $\lambda$ close to zero. In their work they used a stationary higher-dimensional version of the Bounded Distortion Property (see \cite[Proposition 3.11]{DG}). In a later work \cite{M} M. Mei was able to extend that result for irrational $\alpha$ with eventually periodic continuous fraction expansion using Theorem \ref{thm:Sturm}. This case can be boiled down to application of one map $T = T_{a_k} \circ T_{a_{k - 1}} \circ \ldots \circ T_{a_1}$, where $(a_1, a_2, \ldots, a_k)$ stands for the period of the continuous fraction expansion of $\alpha$. For related results see also \cite{G}. We expect that our nonstationary version of the Bounded Distortion Property joined with Theorem \ref{thm:Sturm} can be used to apply these techniques to $\alpha$ with aperiodic continuous fraction expansion to study spectral properties of a wider class of operators $H_{\alpha, \lambda, \omega}$.

\section*{Acknowledgements}

We are grateful to A. Gorodetski for introducing us to the problem. The second author was supported in part by NSF grant DMS--2247966 (PI: A.\,Gorodetski).

\end{document}